\documentclass{article}
 \usepackage{graphicx}
\usepackage{amsthm}
\usepackage{amsmath}
\usepackage{mathrsfs}
\usepackage[all,cmtip]{xy}
\usepackage{amssymb}
\usepackage{amsthm}
\usepackage{fancyhdr}

\begin{document}

\title{Stably free modules over virtually free groups\footnote{The final publication is available at www.springerlink.com DOI: 10.1007/s00013-012-0432-9}}
\author{Seamus O'Shea}
\date{}
\maketitle

\begin{abstract}
Let $F_m$ be the free group on $m$ generators and let $G$ be a finite nilpotent group of non square-free order; we show that for each $m\ge2$ the integral group ring ${\bf Z}[G\times F_m]$ has infinitely many stably free modules of rank 1. 
\end{abstract}



 \newtheorem{thm}{Theorem}[section]
 \newtheorem{cor}[thm]{Corollary}
 \newtheorem{lem}[thm]{Lemma}
 \newtheorem{prop}[thm]{Proposition}
 \theoremstyle{definition}
 \newtheorem{defn}[thm]{Definition}
 \theoremstyle{remark}
 \newtheorem{rem}[thm]{Remark}
 \newtheorem*{ex}{Example}

\section{Introduction}
We study stably free modules over the integral group ring ${\bf Z}[G \times F_m]$,
where $G$ is a finite group and $F_m$ is the free group on $m$
generators. A finitely generated module $P$ over a ring $\Lambda$ is said to be stably free when there exists a natural number $n$ such that $P \oplus \Lambda^n \cong \Lambda^m$ for some $m$. Provided that $\Lambda$ satisfies the invariant basis number property (see \cite{cohn2}), we may uniquely define the rank of $P$ to be $m-n$. Every integral group ring has the invariant basis number property.  We shall prove:

\begin{thm}\label{A}
Let $G$ be a finite nilpotent group of non square-free order. Then there are
infinitely many non-isomorphic stably free modules of rank 1 over  ${\bf
Z}[G \times F_m]$ when $m\ge 2$.
\end{thm}

In contrast Johnson \cite{feaj} has shown that both ${\bf Z}[C_p \times F_m]$ and
${\bf Z}[D_{2p} \times F_m]$ admit no non-free stably free modules when $p$ is
prime, $C_p$ is the cyclic group of order $p$ and $D_{2p}$ is the dihedral group of
order $2p$. Johnson \cite{feajfield} has also shown that $k[G \times F_m]$ admits no non-free stably free modules when $k$ is any field and $G$ is any finite group.

\pagestyle{plain}
In order to produce examples of non-free stably free modules we will use techniques of
Milnor \cite{milnor} for constructing projective modules over fibre product rings.
Denote by $\mathcal{SF}_1(\Lambda)$ the set of isomorphism classes of stably free
modules of rank 1 over a ring $\Lambda$. Our main theorem will be deduced from the
following two special cases:

\textbf{(I)}: $\mathcal{SF}_1({\bf Z}[C_{p^2}\times F_m])$ is infinite for every
prime $p$ and $m\ge2$;

\textbf{(II)}: $\mathcal{SF}_1({\bf Z}[C_{p}\times C_p\times F_m])$ is infinite for
every prime $p$ and $m\ge2$.

\section{Constructing stably free modules}

In \cite{milnor}, Milnor introduced techniques for analysing the structure of
projective modules over a fibre product ring in terms of its factors. These
techniques were further developed by Swan in \cite{swan1} to investigate the
structure of stably free modules over various group rings. Consider a commutative
square of ring homomorphisms 

\begin{equation}\label{mils}
\mathcal{A}\ = \left\{
\vcenter{\xymatrix{A\ar[r]^{\pi_-}\ar[d]^{\pi_+}&A_-\ar[d]^{\psi_-}\\
A_+\ar[r]^{\psi_+}&A_0\\}}\right.
\end{equation}
such that: (i) $A$ is the fibre product of $A_-$ and $A_+$ over $A_0$ :
$A=A_-\times_{A_0} A_+$; (ii) $\psi_+$ is surjective. Such a square will be called a
\emph{Milnor square}.

Any right module $M$ over $A$ determines a triple $(M_+,M_-;\alpha(M))$ where
$M_\sigma = M\otimes_{\pi_\sigma} A_\sigma$ for $\sigma = +, -$ and $\alpha(M):M\otimes_{\psi_+\pi_+}
A_0\to M\otimes_{\psi_-\pi_-} A_0$ is an $A_0$-module isomorphism. Conversely, any triple $(M_+,M_-;\alpha)$ 
with $M_\sigma$ a (right) $A_\sigma$ module for $\sigma =+,-$ and $\alpha$ an isomorphism $\alpha:M_+ \otimes_{\psi_+} A_0 \to M_-\otimes_{\psi_-} A_0$ determines an $A$ module given explicitly as  
$$
\langle M_+, M_-, \alpha \rangle = \{ (m_+, m_-)\in M_+ \times M_- \ | \ \alpha(m_+\otimes 1 ) = m_-\otimes 1 \}.
$$
The $A$-action on $\langle M_+, M_-, \alpha \rangle$  is then given by $$ (m_+,m_-)\cdot a = (m_+\cdot\pi_+(a), m_-\cdot\pi_-(a)).$$

$M$ is said to be \emph{locally free} if both $M_+$ and $M_-$ are free; since $A_+\otimes A_0\cong A_-\otimes A_0\cong A_0$ the rank of $M_+$ and $M_-$ are necessarily the same, and we say that the rank of $M$ is this common rank. A locally free module over $A$ is automatically projective (see \cite{milnor}); the
question of when it is stably free is rather more delicate. Denote the set of
isomorphism classes of finitely generated locally free modules of rank $n$ over $A$
by $\mathcal{LF}_n(A)$. The triple $(M_+,M_-;\alpha(M))$ associated to $M$ does not
completely determine $M$; however by (\cite{swan1}, Lemma A4) there exists a
bijection
\begin{equation}\label{oneone}
\mathcal{LF}_n(A) \leftrightarrow \psi_-(GL_n(A_-))\backslash
GL_n(A_0)/\psi_+(GL_n(A_+))
\end{equation}
Abbreviate the space of double cosets on the right to $\overline{GL_n}(\mathcal{A})$. For each
pair of integers $n, k \ge 1$ define a stabilization map 

\begin{align*}
s_{n,k}:\overline{GL_n}(\mathcal{A})&\to \overline{GL_{n+k}}(\mathcal{A})\\
[\alpha]&\mapsto [\alpha\oplus I_k]
\end{align*}
Then since the free module $A^n$ determines the triple $(A_+^n,A_-^n;I_n)$, we have that
$M$ is stably free if and only if $\alpha = \alpha(M)$ satisfies $s_{n,k}(\alpha)=[I_{n+k}]$ for
some $k$.

The following proposition allows us to construct the original examples of non-trivial
stably free modules claimed in \textbf{(I)} and \textbf{(II)} of the introduction.

\begin{prop}\label{isf}
Let  $\mathcal{A}$ be a Milnor square as in (\ref{mils}) and suppose that 
\begin{equation}
\psi_-(A_-^*)\backslash [A_0^*,A_0^*]/\psi_+(A_+^*)
\end{equation} is infinite, where $A_\sigma^*$ denotes the unit subgroup of $A_\sigma$ and $[ A^*_0, A^*_0 ]$ denotes the commutator subgroup of $A_0^*$.
Then $\mathcal{SF}_1(A)$ is infinite.
\end{prop}

\begin{proof}
Let $\{a_i\}_{i\in I}$ be an infinite set of coset representatives in 
$\psi_-(A_-^*)\backslash [A_0^*,A_0^*]/\psi_+(A_+^*)$. Each $a_i$ determines an automorphism $(a_i)\in GL_1(A_0)$ and so for each $i\in I$ we may form the locally free
$A$-module $P_i=\langle A_+,A_-;(a_i)\rangle$. Then clearly by (\ref{oneone}) $P_i\cong P_j$ if and only if $i=j$. To
see that each $P_i$ is stably free, consider $s_{1,1}([a_i])=\left[\begin{array}{cc} a_i&0\\
0&1 \end{array}\right]$. Each $a_i\in [A_0^*,A_0^*]$ and so by Whitehead's lemma
$s_{1,1}([a_i])\in E_2(A_0)$, where $E_2(A_0)$ is the subgroup of $GL_2(A_0)$ generated by the elementary matrices $E(i,j;a)=I_2+a\epsilon(i,j)$ ($a\in A_0$). Since $\psi_+ :A_+\to A_0$ is surjective, we have an inclusion 
$E_2(A_0)= \psi_+E_2(A_+)\subset \psi_+GL_2(A_+)$ and therefore $s_{1,1}([a_i])= [I_2] \in
\overline{GL_2}(\mathcal{A})$ and so $P_i\oplus A \cong A^2$.

\end{proof}

\section{Stably free cancellation and generalized Euclidean rings}

A ring $\Lambda$ is said to have stably free cancellation (SFC) when every stably
free module over $\Lambda$ is actually free. All principal ideal domains have SFC,
as do all local rings. Much of the following discussion is due to Johnson \cite{feaj}.

For any ring $\Lambda$, denote by $\mbox{rad}(\Lambda)$ the Jacobson radical of
$\Lambda$. Recall that an ideal {\bf m} of $\Lambda$ is said to be \emph{radical} when
${\bf m}\subset \mbox{rad}(\Lambda)$.
It is a clear consequence of (\cite{bass}, p.90, Prop. 2.12) that:

\begin{prop}\label{bour}
Let {\bf m} be a two sided radical ideal in $\Lambda$. Then 
\[
\Lambda/{\bf m}\ \mbox{has SFC} \implies \Lambda\ \mbox{has SFC}
\]
\end{prop}

Let $D$ be a (possibly non-commutative) division ring. Dicks and Sontag \cite{dicks} have shown that $D[F_m]$ has SFC. As a consequence of Morita equivalence, a matrix ring $M_n(\Lambda)$ has SFC if and only if $\Lambda$ has SFC; applying Wedderburn's theorem now shows that $\Lambda[F_m]$ has SFC for any left semi-simple ring $\Lambda$. (Note that a product $\Lambda= \Lambda_1\times \Lambda_2$ has SFC if and only if both $\Lambda_1 $ and
$\Lambda_2$ have SFC.) Now suppose that $\Lambda$ is a left artinian ring. The
canonical mapping $\phi:\Lambda \to \Lambda/\mbox{rad}(\Lambda)$ induces a
surjective ring homomorphism $\phi_*: \Lambda[F_m]\to \Lambda/\mbox{rad}(\Lambda)[F_m]$ in which
$\mbox{ker}(\phi_*) = \mbox{rad}(\Lambda)[F_m]$. Since $\Lambda$ is left artinian, $\mbox{rad}(\Lambda)$ is nilpotent (see Lam \cite{lam}, Theorem 4.12) and hence $\mbox{rad}(\Lambda)[F_m]$ is a radical ideal in $\Lambda[F_m]$. Applying \ref{bour} with ${\bf m}=\mbox{rad}(\Lambda)[F_m]$ now shows:

\begin{cor}\label{fin}
Let $\Lambda$ be a left artinian ring. Then $\Lambda[F_m]$ has SFC.
\end{cor}

For any ring $\Lambda$ denote by $E_n(\Lambda)$ the subgroup of $GL_n(\Lambda)$
generated by the elementary matrices $E(i,j;\lambda) = I_n + \lambda \epsilon(i,j)$
$\ (\lambda\in \Lambda)$. Denote by $D_n(\Lambda)$ the subgroup of $GL_n(\Lambda)$
consisting of all diagonal matrices. $\Lambda$ is said to be \emph{generalized Euclidean}
when, for all $n\ge 2$ the following statement holds: for all $A\in GL_n(\Lambda)$ there exists $E\in
E_n(\Lambda)$ and $D\in D_n(\Lambda)$ such that $A=DE$; in other words every
invertible matrix over $\Lambda$ is reducible to a diagonal matrix by means of
elementary row and column operations. The notion of a generalized Euclidean ring is often useful when dealing with modules over fibre product rings.

A ring homomorphism $\phi: \Lambda_1 \to \Lambda_2$ induces a mapping $\phi : 
\mathcal{SF}_1(\Lambda_1) \to \mathcal{SF}_1(\Lambda_2)$ given by
$\phi(S)=S\otimes_{\phi}\Lambda_2$. The following is proven in \cite{feaj}:

\begin{prop}\label{up}
Let $A$ be a Milnor square. If
$A_0$ has SFC and is generalized Euclidean, then the induced map $\pi_+\times \pi_- :
\mathcal{SF}_1(A) \to \mathcal{SF}_1(A_+)\times \mathcal{SF}_1(A_-)$ is surjective.
\end{prop}

A well known theorem of Cohn \cite{cohn} states that $k[F_m]$ is generalized Euclidean whenever $k$ is a
division ring. If $\Lambda$ is generalized Euclidean, then so is the matrix ring $M_n(\Lambda)$. Since generalized Euclidean rings are closed under products, we have that $\Lambda[F_m]$ is generalized Euclidean whenever $\Lambda$ is semi-simple.

\begin{prop}\label{localcohn}
If $\Lambda$ is a ring such that $\Lambda/I$ is generalized Euclidean for some nilpotent ideal $I$, then $\Lambda$ is also generalized Euclidean. 
\end{prop}
\begin{proof}
Let $A\in GL_n(\Lambda)$ and consider the matrix $\psi_*(A)\in GL_n(\Lambda/I)$, where $\psi_*:GL_n(\Lambda) \to GL_n(\Lambda/I)$ is induced by the mapping $\psi : \Lambda \to \Lambda/I$. Then by hypothesis we may write $\psi_*(A) = D E$ for some $D\in D_n(\Lambda/I)$ and some $E\in E_n(\Lambda/I)$. Choose $\hat{D}\in D_n(\Lambda)$ and $\hat{E}\in E_n (\Lambda)$ such that $\psi_*(\hat{D})=D^{-1}$ and $\psi_*(\hat{E}) = E^{-1}$; if $X =A\hat{E}\hat{D}$ then clearly $\psi_*(X) = I_n$. 

Since $\psi(X_{nn})= 1 $, we have that $X_{nn}$ is a unit as $I$ is nilpotent. Therefore by means of elementary row and column operations we may reduce $X$ so that $X_{rn}=X_{nr}=0$ for $r\neq n$.  Repeating this operation for each diagonal element of $X$ we can reduce $X$, and hence $A$, to a diagonal matrix as required.
\end{proof}

As noted above, $\mbox{rad}(\Lambda)$ is nilpotent whenever $\Lambda$ is left artinian. Applying \ref{localcohn} with $I = \mbox{rad}(\Lambda)[F_m]$ gives:
\begin{cor}\label{fin2}
Let $\Lambda$ be a left artinian ring. Then $\Lambda[F_m]$ is generalized Euclidean.
\end{cor}

\section{Proof of I and II}
It is easy to show that, if $I$ and $J$ are ideals of a ring $R$, then 
$$
\xymatrix{R/(I\cap J) \ar[r]\ar[d] & R/J \ar[d]\\
R/I \ar[r] & R/(I + J)}
$$
is a fibre square.
For any positive integer $d$ let $c_d(x)$ denote the $d^{th}$ cyclotomic polynomial. From the factorization $(x^{p^2}-1) = c_{p^2}(x)c_p(x)c_1(x) = c_{p^2}(x)(x^p -1)$ we obtain the Milnor square
$$
\xymatrix{ {\bf Z}[x]/(x^{p^2}-1)\ar[r]\ar[d]& {\bf Z}[x]/(c_{p^2}(x))\ar[d] \\
{\bf Z}[x]/(x^p-1)\ar[r] & {\bf Z}[x]/I}
$$
where $I$ is the sum of the ideals $(x^p-1)$ and $(c_{p^2}(x))$. However, since  $c_{p^2}(x) = (x^{p(p-1)}+x^{p(p-2)}+\ldots+x^p+1)$, we have
$$
p = c_{p^2}(x) - (x^{p(p-2)}+2x^{p(p-3)}+\ldots+(p-2)x^p+(p-1))(x^p - 1),
$$
and hence $I = ( p , x^p -1)$. Therefore we may rewrite the above square as 
$$
\xymatrix{ {\bf Z}[x]/(x^{p^2}-1)\ar[r]\ar[d]& {\bf Z}[x]/(c_{p^2}(x))\ar[d] \\
{\bf Z}[x]/(x^p-1)\ar[r] & {\bf F}_p[x]/(x^p-1)}
$$
Writing $A[C_n]=A[x]/(x^n-1)$ and applying the functor $-\otimes {\bf Z}[F_m]$ we
obtain another Milnor square: 
\begin{equation*}
\mathcal{A}= \left\{
\vcenter{\xymatrix{ {\bf Z}[C_{p^2}\times F_m]\ar[r]\ar[d]& {\bf
Z}[\zeta_{p^2}][F_m]\ar[d]^{\psi_-} \\
{\bf Z}[C_p\times F_m]\ar[r]^{\psi_+} & {\bf F}_p[C_p\times F_m]}}\right.
\end{equation*}
where $\zeta_{p^2}$ is a primitive $p^2$-th root of unity.
Since ${\bf Z}[\zeta_{p^2}]$ is an integral domain, ${\bf Z}[\zeta_{p^2}][F_m]$ has
only trivial units (see \cite{passman}, p.591 and p.598); that is 
$$
{\bf Z}[\zeta_{p^2}][F_m]^*= ({\bf Z}[\zeta_{p^2}])^* \times F_m
$$
\begin{prop}\label{unit}
${\bf Z}[C_p\times F_m]^* = ({\bf Z}[C_p])^* \times F_m$
\end{prop}
\begin{proof}
Consider the following fibre square, which arises from the factorization $(x^p-1)=(x-1)c_p(x)$:
$$
\xymatrix{ {\bf Z}[C_{p}\times F_m]^*\ar[r]^{\pi_-}\ar[d]^{\pi_+}& {\bf
Z}[\zeta_{p}][F_m]^*\ar[d] \\
{\bf Z}[F_m]^*\ar[r] & {\bf F}_p[F_m]^*}
$$
Let $u\in  {\bf Z}[C_{p}\times F_m]^*$. Then $\pi_+(u)\in {\bf Z}[F_m]^*$ which has
only trivial units; thus 
$$
u= aw\ +\sum_{g\in F_m-\{w\}}a_gg
$$
where $a(1)=\pm1$, $w\in F_m$ and $a_g\in {\bf Z}[C_p]={\bf Z}[x]/(x^p-1)$. Each
$a_g$ is divisible by $(x-1)$ since $a_g\in \ker(\pi_+)$. Now consider
$$
\pi_-(u) = a(\zeta_p)w\ +\sum_{g\in F_m-\{w\}}a_g(\zeta_p)g.
$$ 
We cannot have $a(\zeta_p)=0$, or else $(1+x+\ldots+x^{p-1})|a \implies p | a(1)$ which is a contradiction. Therefore, since ${\bf Z}[\zeta_{p}][F_m]$ has only trivial units we must have
$a_g(\zeta_p)=0$ for all $g\in F_m - \{w\}$. Therefore both $(1+x+\ldots+ x^{p-1})$ and $(x-1)$ divide each
$a_g$ and so each $a_g=0$. 
\end{proof}

\begin{prop}\label{infi}
If $m\ge 2$ then $$X=\psi_-({\bf Z}[\zeta_{p^2}][F_m]^*)\backslash[{\bf F}_p[C_p\times
F_m]^*,{\bf F}_p[C_p\times F_m]^*]/\psi_+({\bf Z}[C_p\times F_m]^*)$$ is infinite.
\end{prop}
\begin{proof}
Let $x$ be a generator of $C_p$ and define $y=(1-x)\in {\bf F}_p[C_p]$; then $y^p=0$. Let $s$ and $t$ be two generators of $F_m$ and define 
$$
\delta_n = (1+yt)s^n(1+yt)^{-1}s^{-n}\in [{\bf F}_p[C_p\times F_m]^*,{\bf
F}_p[C_p\times F_m]^*].
$$
We claim that $\{\delta_n \ | \ n\in {\bf N}\}$ is a set of distinct coset
representatives in $X$. Suppose that $[\delta_n]=[\delta_m]$; then there exists
$u\in {\bf Z}[\zeta_{p^2}][F_m]^*$ and $u'\in {\bf Z}[C_p\times F_m]^*$ such that
$\delta_n = \psi_-(u)\delta_m\psi_+(u')$. In fact since $u$ and $u'$ are necessarily
trivial units
$$
\delta_n=\psi_-(a)\psi_+(b)w\delta_m v
$$
for some $a\in {\bf Z}[\zeta_{p^2}]^*$, $b\in {\bf Z}[C_p]^*$ and some $w, v \in F_m$.
The units of ${\bf F}_p[C_p]$ are of the form $c+dy$ where $c\in {\bf F}_p^*$ and
$d\in {\bf F}_p$, as ${\bf F}_p[C_p]$ is a local ring with maximal ideal generated by
$y$. Therefore we have
$$
(1+yt)s^n(1+yt)^{-1}s^{-n}=(c+dy)w(1+yt)s^m(1+yt)^{-1}s^{-m}v.
$$
Expanding both sides and comparing coefficients of $y^0$ we have: $d=0$, $c=1$ and
$w^{-1}=v$. Comparing coefficients of $y^1$ gives
$$
t-s^nts^{-n}=wtw^{-1} - ws^mts^{-m}w^{-1},
$$
and so we must have 
$$
t=wtw^{-1}\ \mbox{and} \ s^nts^{-n}=ws^mts^{-m}w^{-1}.
$$
The first equation shows that $w=1(=v)$ and the second shows that $m=n$; therefore the $\delta_i$ form a set of distinct coset representatives. 
\end{proof}
Putting (\ref{infi}) and (\ref{isf}) together proves ({\bf I}) of the introduction:

\begin{prop}\label{Ii}
For every prime $p$ and every $m\ge 2$, $\mathcal{SF}_1({\bf Z}[C_{p^2}\times F_m])$
is infinite.
\end{prop}

The proof of ({\bf II}) is very similar. Let $A$ be a ring and consider the Milnor
square
\begin{equation}\label{square2}
\vcenter{\xymatrix{A[x]/(x^p-1)\ar[d]\ar[r]& A[x]/(1+x+\ldots +x^{p-1})\ar[d]\\
A\ar[r]& A/p}}
\end{equation}
Setting $A={\bf Z}[y]/(y^p-1)$ we have
$$\xymatrix{{\bf Z}[x,y]/(x^p-1,y^p-1)\ar[d]\ar[r]& {\bf Z}[x,y]/(\Sigma_x,y^p-1
)\ar[d]\\
{\bf Z}[y]/(y^p-1)\ar[r]& {\bf F}_p[y]/(y^p-1)}
$$
where $\Sigma_x = 1 + x+\ldots x^{p-1}$.
Making the identifications ${\bf Z}[x,y]/(x^p-1,y^p-1)={\bf Z}[C_p\times C_p]$,
${\bf Z}[x,y]/(\Sigma_x,y^p-1)= {\bf Z}[\zeta_p][C_p]$ and $B[y]/(y^p-1)=B[C_p]$ and
tensoring with ${\bf Z}[F_m]$ we have
$$\mathcal{B}=\left\{\vcenter{\xymatrix{{\bf Z}[C_p\times C_p\times
F_m]\ar[d]\ar[r]& {\bf Z}[\zeta_p][C_p\times F_m]\ar[d]^{\phi_-}\\
{\bf Z}[C_p\times F_m]\ar[r]^{\phi_+}& {\bf F}_p[C_p\times F_m]}}\right.
$$
We first need to show that ${\bf Z}[\zeta_p][C_p\times F_m]$ has only trivial units:

\begin{prop}\label{units2}
${\bf Z}[\zeta_p][C_p\times F_m]^*= ({\bf Z}[\zeta_p][C_p])^*\times F_m$.
\end{prop}
\begin{proof}
Consider the Milnor square formed by setting $A={\bf Z}[y]/(1+y+\ldots + y^p)={\bf Z}[y]/(\Sigma_y)$ in
(\ref{square2}), tensoring with ${\bf Z}[F_m]$ and then taking unit groups:
$$\xymatrix{{\bf Z}[x,y]/(x^p-1,\Sigma_y)[F_m]^*\ar[d]\ar[r]& {\bf
Z}[x,y]/(\Sigma_x,\Sigma_y)[F_m]^*\ar[d]\\
{\bf Z}[y]/(\Sigma_y)[F_m]^*\ar[r]& {\bf F}_p[y]/(\Sigma_y)[F_m]^*}
$$
Since both ${\bf Z}[x,y]/(\Sigma_x,\Sigma_y)$ and ${\bf Z}[y]/(\Sigma_y)$ are
integral domains the corresponding corners have only trivial units. A similar proof
to that of (\ref{unit}) now applies.
\end{proof}
Essentially the same proof as that of (\ref{infi}) proves:

\begin{prop}\label{infi2} If $m\ge 2$ then
$$\phi_-({\bf Z}[\zeta_p][C_p\times F_m]^*)\backslash[{\bf F}_p[C_p\times F_m]^*,{\bf
F}_p[C_p\times F_m]^*]/\phi_+({\bf Z}[C_p\times F_m]^*)$$ is infinite.
\end{prop}
Together (\ref{infi2}) and (\ref{isf}) prove ({\bf II}) of the introduction:

\begin{prop}\label{II}
For every prime $p$ and every $m\ge 2$, $\mathcal{SF}_1({\bf Z}[C_{p}\times
C_p\times F_m])$ is infinite.
\end{prop}

\section{Proof of main theorem}

Let $G$ be a finite group and let $H$ be a normal subgroup of $G$.  We may form the
Milnor square:
$$\xymatrix{{\bf Z}[G] \ar[r]\ar[d]&{\bf Z}[G]/(\Sigma_H)\ar[d]\\
{\bf Z}[G/H]\ar[r] &({\bf Z}/N)[G/H]}
$$
where $\Sigma_H = \sum_{h\in H}h$ and $N = | H |$. Tensoring with ${\bf Z}[F_m]$ we
have:
$$\xymatrix{{\bf Z}[G\times F_m] \ar[r]\ar[d]&{\bf Z}[G\times F_m]/(\Sigma_H)\ar[d]\\
{\bf Z}[G/H\times F_m]\ar[r] &({\bf Z}/N)[G/H\times F_m]}
$$
Now by (\ref{fin}) and (\ref{fin2}), $({\bf Z}/N)[G/H\times F_m]$ has SFC and is
generalized Euclidean. Hence by (\ref{up}) the induced map 
$$\mathcal{SF}_1({\bf
Z}[G\times F_m]) \to \mathcal{SF}_1({\bf Z}[G/H\times F_m])\times
\mathcal{SF}_1({\bf Z}[G\times F_m]/(\Sigma_H))$$
 is surjective and thus:

\begin{prop}\label{C}
Let $G$ be a finite group with normal subgroup $H\lhd G$. If $\mathcal{SF}_1({\bf
Z}[G/H\times F_m])$ is infinite then so is $\mathcal{SF}_1({\bf Z}[G \times F_m])$.
\end{prop}

Let $G$ be a finite group of order $p^k$ where $p$ is prime and $k\ge2$. Then there
exists a normal subgroup $H \lhd G$ such that $|H|=p^{k-2}$ (see \cite{fgt} p.24).
Hence either $G/H \cong C_{p^2}$ or $G/H\cong C_p\times C_p$; in either case by
(\ref{Ii}), (\ref{II}) and (\ref{C}) $\mathcal{SF}_1({\bf Z}[G\times F_m])$ is
infinite. 

Now let $G$ be a finite nilpotent group of non square-free order. Since $G$ is
nilpotent, $G$ is the direct product of its Sylow subgroups (see \cite{fgt}, p.24)
say $G\cong H_1\times \ldots \times H_r$. As $|G|$ is non square-free we may choose
a prime $p$ such that $p^k$ is the largest power of $p$ dividing $|G|$ and where
$k\ge2$. Therefore at least one of the $H_i$ has order $p^k$ --- assume without loss
of generality that $|H_1|= p^k$. Then $|G/(H_2\times \ldots \times H_r)| = p^k$ and
so by (\ref{C}) this completes the proof of our main theorem (\ref{A}).

\subsection*{Acknowledgment}

The author wishes to thank his supervisor, Professor F. E. A. Johnson both for suggesting the problem and for many invaluable discussions.

\noindent {\bf Department of Mathematics,\\ University College London,\\ Gower Street,\\ London, U.K.\\ email: s.o'shea@ucl.ac.uk}
\end{document}